\documentclass[letterpaper]{amsart}
\usepackage{amsmath}
\usepackage{amstext}
\usepackage{amssymb}
\usepackage{amsfonts}
\usepackage{enumerate}
\usepackage{braket}
\usepackage{color}
\usepackage[all]{xy}
\usepackage{url}
\usepackage{bbm}
\usepackage{mathtools}

\usepackage{hyperref,cleveref,graphics,mathrsfs}

\numberwithin{equation}{section}

\newtheoremstyle{note}
{1em}
{1em}
{}
{}
{\bfseries}
{: }
{.5em}
{}

\newtheorem{theorem}{Theorem}[section]
\newtheorem{lemma}[theorem]{Lemma}
\newtheorem{proposition}[theorem]{Proposition}
\newtheorem{corollary}[theorem]{Corollary}

\theoremstyle{note}
\newtheorem{remark}[theorem]{Remark}

\newtheorem{definition}[theorem]{Definition}
\newtheorem{example}[theorem]{Example}

\newtheorem{claim}[theorem]{Claim}

\newcommand{\N}{{\mathbb{N}}}
\newcommand{\R}{{\mathbb{R}}}
\newcommand{\C}{{\mathbb{C}}}
\newcommand{\Q}{{\mathbb{Q}}}

\newcommand{\n}[1]{ \left\|#1\right\| }
\newcommand{\tn}[1]{{\left\vert\kern-0.25ex\left\vert\kern-0.25ex\left\vert #1 
    \right\vert\kern-0.25ex\right\vert\kern-0.25ex\right\vert}}

\newcommand{\eps}{\varepsilon}

\newcommand{\M}{{\mathrm{M}}}

\newcommand{\cB}{{\mathcal{B}}}

\newcommand{\vertiii}[1]{{\left\vert\kern-0.25ex\left\vert\kern-0.25ex\left\vert #1 
    \right\vert\kern-0.25ex\right\vert\kern-0.25ex\right\vert}}

\DeclareMathOperator{\cb}{cb}

\DeclareMathOperator{\Lip}{Lip}

\title[On the small scale nonlinear theory of operator spaces]{On the small scale nonlinear theory of operator spaces}

\author[B.\ M.\ Braga]{Bruno M.\ Braga}

\address[B.\ M.\ Braga]{IMPA, Estrada Dona Castorina 110, 22460-320, Rio de Janeiro, Brazil}
\email{demendoncabraga@gmail.com}
\urladdr{https://sites.google.com/site/demendoncabraga/}
\thanks{B. M. Braga  was partially supported by FAPERJ (Proc. E-26/200.167/2023) and by CNPq (Proc. 303571/2022-5). 
J. A. Ch\'avez-Dom\'inguez  was partially supported by NSF grants DMS-1900985 and DMS-2247374.}

\author[J.\ A.\ Ch\'avez-Dom\'inguez]{Javier Alejandro Ch\'avez-Dom\'inguez}
\address[J.\ A.\ Ch\'avez-Dom\'inguez]{Department of Mathematics, University of Oklahoma, Norman, OK 73019-3103,
USA} \email{jachavezd@ou.edu}

\subjclass[2010]{Primary:  47L25, 46L07, 46B80} 

\begin{document}
\maketitle

\begin{abstract}
We initiate the study of the small scale geometry of operator spaces.
The authors have previously shown that a map between operator spaces which is completely coarse (that is, the sequence of its amplifications is equi-coarse) must be $\R$-linear.
We obtain a generalization of the aforementioned result to completely coarse maps defined on the unit ball of an operator space.
By relaxing the condition to a small scale one, we prove that there are many non-linear examples of maps which are completely Lipshitz in small scale.
We define a geometric parameter for homogeneous Hilbertian operator spaces which imposes restrictions on the existence of such maps.
\end{abstract}

\section{Introduction}\label{SectionIntro}

Although the interest of functional analysts  in the nonlinear theory of Banach spaces has increased significantly in the past few decades (e.g.\ \cite{JohnsonLindenstraussSchechtman1996GAFA,GodefroyKalton2003,MendelNaor2008AnnalsMath,BaudierLancienSchlumprecht2018JAMS}), researchers have only recently started to develop the nonlinear theory of their noncommutative counterpart, i.e., of operator spaces (see \cite{BragaChavezDominguez2020PAMS,BragaChavezDominguezSinclair2022MathAnn,Braga2021OpSp,BragaOikhberg2023MathZ}). As of now,  this study was restricted to constructing a  large scale geometry for such spaces. The goal of the current article is to initiate the treatment of the small scale geometry of operator spaces. 

Before describing our main results, we start this introduction with a paragraph recalling the basics for the non-expert:
  an \emph{operator space} is a Banach subspace of the space of bounded operators on a given Hilbert space $H$, which we denote by  $\cB(H)$. Given $n\in\N$ and a set $X$, we denote the space of $n$-by-$n$ matrices with entries in $X$ by $\M_n(X)$ --- if $X$ is either a vector space or an algebra, $\M_n(X)$ inherits a canonical vector space or algebra structure, respectively. Since $\M_n(\cB(H))$ is   canonically isomorphic to $\cB(H^{\oplus n})$ (where $H^{\oplus n}$ denotes  the Hilbert sum of $n$ copies of $H$), each $\M_n(\cB(H))$ is endowed with the canonical  norm  given by this isomorphism. Given an operator space $X\subseteq \cB(H)$, the inclusions $\M_n(X)\subseteq \M_n(\cB(H))$ then induce norms on each $\M_n(X)$. The \emph{$n$-amplification} of a map $f\colon X\to Y$ between operator spaces is the map $f_n\colon \M_n(X)\to \M_n(Y)$ given by 
\[f_n([x_{ij}])=[f(x_{ij})]\ \text{ for all }\ [x_{ij}]\in \M_n(X).\]
If $f$ is linear, so is each $f_n$ and  $\|f_n\|_n$ denotes its operator norm. The \emph{completely bounded norm of $f$}, abbreviated as the \emph{cb-norm of $f$}, is given by 
\[\|f\|_{\mathrm{cb}}=\sup_{n\in\N}\|f_n\|\]
and $f$ is called \emph{completely bounded} if $\|f\|_{\mathrm{cb}}<\infty$. Completely bounded maps play the role bounded maps play in Banach space theory and are used to define \emph{complete isomorphisms} between operator spaces. 

Our approach to study the small scale geometry of operator spaces comes from   a strengthening of the main result of \cite{BragaChavezDominguez2020PAMS}; which, as the reader will see below, is something in between the large and the small scale geometry of operator spaces. We start recalling the concept of coarse maps in the category of operator spaces:

\begin{definition}
Let $X$ and $Y$ be operator spaces and $B\subseteq X$. A map $f\colon B\to Y$ is called \emph{completely coarse} if for all $r>0$ there is $s>0$ such that 
\[\|[x_{ij}]-[y_{ij}]\|_{\M_n(X)}\leq r\ \text{ implies }\ \|f_n([x_{ij}])-f_n([y_{ij}])\|_{\M_n(Y)}\leq s\]
for all $n\in\N$ and all $[x_{ij}],[y_{ij}]\in \M_n(B)$.\footnote{In classic coarse geometry, a map is called \emph{coarse} if this holds for $n=1$. }
\end{definition}

The main result of \cite{BragaChavezDominguez2020PAMS} showed  that, despite its nonlinear definition, completely coarse maps are essentially already linear. Precisely, the following version of the Mazur-Ulam theorem holds for completely coarse maps between operator spaces:  

\begin{theorem}\emph{(}\cite[Theorem 1.1]{BragaChavezDominguez2020PAMS}\emph{).}
    Let $X$ and $Y$ be operator spaces. Any completely coarse map $f\colon X\to  Y$
with $f(0)=0$ must be $\R$-linear. \end{theorem}

In this paper, we take the techniques developed in \cite{BragaChavezDominguez2020PAMS} further and  show that a much stronger result remains valid. Throughout these notes, if $X$ is a Banach space,  $B_X$ denotes its closed unit ball.

\begin{theorem}\label{ThmCompleteCoarseAreRLinear}
    Let $X$ and $Y$ be operator spaces. Any completely coarse map $f\colon B_X\to  Y$
with $f(0)=0$ must be the restriction of an $\R$-linear map. 
\end{theorem}

The previous result exhausts any possible attempt to build a large scale nonlinear theory for operator spaces in the ``naive'' way: ideally, if we were to merge operator space theory with the nonlinear theory for Banach spaces, there could be an interesting   theory which would capture  aspects of the large scale geometry of operator spaces by simply considering maps $f\colon X\to Y$, or at least maps $f\colon B_X\to Y$, and their amplifications. However, this does not mean that it is not possible to obtain a nontrivial large scale geometry of operator spaces. Indeed, as shown in \cite{BragaChavezDominguezSinclair2022MathAnn,BragaOikhberg2023MathZ,Braga2021OpSp}, there are several interesting things to be said if one takes a more sophisticated approach: for instance,  instead of considering a \emph{single}  map $f$ from either $X $ or $B_X$ to $Y$, one can consider \emph{sequences} of maps $(f^n\colon X\to Y)_n$ or even  of maps $(f^n\colon n\cdot B_X\to Y)_n$.

Motivated by Theorem \ref{ThmCompleteCoarseAreRLinear},  we start  the study of the   \emph{small} scale structure of operator spaces. For that,  we want to take into  consideration not only the behavior of maps  on sets of small diameter, but also their amplifications on such sets. Notice that this is not   dealt with in Theorem \ref{ThmCompleteCoarseAreRLinear} since $\lim_{n\to \infty}\mathrm{diam}(\M_n(B_X))=\infty$.\footnote{This is why we say Theorem \ref{ThmCompleteCoarseAreRLinear} is somewhat in between the scope of large and small scale geometric analysis.} For this reason, we will restrict our maps to the unit balls of $\M_n(X)$ in order to guarantee the diameters of the sets are uniformly bounded. Before presenting our main findings, we start by recalling the definition of the modulus of uniform continuity of a map. Let $(X,d)$ and $(Y,\partial)$ be metric spaces, and $f\colon X\to Y$ be a map. The \emph{modulus of uniform continuity of $f$} is the map $\omega_f\colon [0,\infty)\to [0,\infty]$ given by 
\[\omega_f(t)=\sup\big\{\partial (f(x),f(y))\mid d(x,y)\leq t\big\}.\]

\begin{definition}
Let $X$ and $Y$ be operator spaces, and  $f\colon B_X\to Y$ be a map. 

\begin{enumerate}
\item The \emph{small scale modulus of uniform continuity of $f$} is the map $\omega_f^\mathrm{ss}\colon [0,\infty)\to [0,\infty]$ given by 
\[\omega_{f}^{\mathrm{ss}}(t)=\sup_{n\in\N}\omega_{f_n\restriction B_{\M_n(X)}}(t)\ \text{  for all }\ t\geq 0.\]
\item We say that $f$ is \emph{completely  Lipschitz in  small scale} if there is $L>0$ such that 
\[\omega_f^{\mathrm{ss}}(t)\leq Lt\ \text{ for all }\ t\geq 0.\] 
\end{enumerate}

\end{definition}

In contrast with Theorem \ref{ThmCompleteCoarseAreRLinear}, the next result shows that  the property of a map being completely  Lipschitz in small scale does not force the map to be the restriction of an $\R$-linear map. In fact, the next theorem  provides a large class of  non-$\R$-linear maps which are  completely  Lipschitz in   small scale.

  \begin{theorem}\label{ThmCompLipOnSmallScaleExistPol}
Any polynomial $p$ in one complex variable is completely Lipschitz in small scale as a map $B_\C \to \C$.
More generally, if $A \subseteq \mathcal{B}(H)$ is an operator algebra (with its induced operator space structure) then $p\colon B_A \to A$ is  completely Lipschitz in small scale.
\end{theorem}

Theorem \ref{ThmCompLipOnSmallScaleExistPol} can be further generalized since it is obtained by looking at the compositions of $m$-linear maps with completely bounded maps. For brevity, we refer the reader to Section \ref{SectionNonlinearSmallScaleMaps} and Theorem \ref{ThmConstructCompLipSmallScale} for further details.

Knowing that   there are plenty of interesting non-$\R$-linear maps which are completely Lipschitz in small scale, we then turn to study what the existence of such maps can tell us about the operator spaces involved.  For that, we need the embedding notion given by maps which are  completely Lipschitz in small scale. Recall, if  $f\colon (X,d)\to (Y,\partial)$ is a map between metric spaces, then the \emph{compression modulus of $f$} is the map  $\rho_f\colon [0,\infty)\to [0,\infty]$  given by 
\[\rho_f(t)=\inf\big\{\partial (f(x),f(y))\mid d(x,y)\geq t\big\}.\]

\begin{definition}
Let $X$ and $Y$ be operator spaces, and   $f\colon B_X\to Y$ be a map.
\begin{enumerate}
    \item The \emph{small scale compression modulus} $\rho^{\mathrm{ss}}_f\colon [0,\infty)\to [0,\infty]$ is given by \[\rho_{f}^{\mathrm{ss}}(t)=\inf_{n\in\N}\rho_{f_n\restriction B_{\M_n(X)}}(t)\ \text{  for all }\ t\geq 0.\]
    \item We say that $f$ is   a \emph{completely Lipschitz in small scale embedding} if it is completely Lipschitz in small scale and there is $L\geq 1$ such that 
\[\rho_{f}^{\mathrm{ss}}(t)\geq \frac{1}{L}t\ \text{ for all }\ t\geq 0.\]
\end{enumerate} 
\end{definition}

With our notion of nonlinear small scale embeddability being established, we now describe a linear property which is preserved under such notion. Recall that an operator space $X$ is called \emph{Hilbertian} if it is (linearly) isometric to a Hilbert space and \emph{homogeneous} if for every linear map $u\colon X\to X$ we have $\|u\|_{\cb}=\|u\|$. Then, if $X$ is a  homogeneous Hilbertian operator space   with $\dim (E) \ge n$, we can unambiguously define 
\[
\kappa_n(X) = \left\|\begin{bmatrix}
e_1&0&\ldots &0\\
e_2&0&\ldots &0\\ 
\vdots& \vdots &\ddots&\vdots\\
e_n&0&\ldots &0
 \end{bmatrix}\right\|_{\M_{n}(X)}.
\]
where $\{e_1, \dotsc, e_n\}$ is an arbitrary orthonormal set in $X$.  Our main results about rigidity of operator spaces with respect to nonlinear embeddings  will be based  on the asymptotic behavior of $\kappa_n(X)$. If $(a_n)_n$ and $(b_n)_n$ are positive real numbers we use the common notation that $a_n\simeq b_n$ meaning that there is $L\geq 1$ such that $a_n/L\leq b_n\leq  L a_n$ for all $n\in\N$. 

The following is our main theorem about the preservation of the linear geometry of operator spaces by completely Lipschitz in small scale embeddings. 

\begin{theorem}\label{ThmLipPreservatinOfKappa}
    Let $X$ and $Y$ be homogeneous Hilbertian operator spaces. If there is a completely Lipschitz  in small scale embedding $B_{X}\to Y$, then $\kappa_n(X)\simeq \kappa_n(Y)$. 
\end{theorem}

Theorem \ref{ThmLipPreservatinOfKappa} is obtained in two stages, one giving the lower estimate and other the upper (see Theorems \ref{ThmLowerBoundAlpha} and \ref{ThmUpperBound}, respectively). For each of these weaker results, the hypotheses are also much weaker. 

In order to obtain applications of Theorem \ref{ThmLipPreservatinOfKappa}, it is important to compute, or at least estimate,  $\kappa_n(X)$ for some operator spaces. Our main source of examples comes from interpolating operator spaces. We refer the reader to Section \ref{SectionComputeKappa} and the references therein for precise definitions. Here, we simply mention that if $X$ and $Y$ are homogeneous Hilbertian  operator spaces and $\theta\in [0,1]$, $(X,Y)_\theta$ denotes the $\theta$-interpolation operator space of $X$ and $Y$. We compute the following ($R$ and $C$ denote the \emph{row} and the \emph{column} operator spaces, respectively,  see Section \ref{SectionComputeKappa}):
\begin{itemize}
    \item $\kappa_n((R,C)_\theta)=n^{\theta/2}$ (Corollary \ref{CorollaryRCKappa}),
    \item $\kappa_n((\min(\ell_2),\max(\ell_2))_{\theta})=n^{\theta/2}$   (Corollary \ref{CorollaryRCKappa}),
    \item $\kappa_n((R\cap C,R+C)_\theta)=n^{\theta/2}$  (Corollary \ref{CoroRcapCR+C}), and 
    \item $\kappa_n(\Phi)\simeq \sqrt{n}$, where $\Phi$ is the Fermionic operator space (Proposition \ref{PropFermion}). 
\end{itemize}
In particular, the computations above allow us to conclude the following: 

   \begin{corollary}\label{CorRC3}
Let $\theta,\gamma\in [0,1]$,
\begin{itemize}
    \item $X\in \{(R,C)_\theta,(\min(\ell_2),\max(\ell_2))_\theta,(R\cap C,R+C)_\theta\}$, and 
    \item $Y\in \{(R,C)_\gamma,(\min(\ell_2),\max(\ell_2))_\gamma,(R\cap C,R+C)_\gamma\}$.
\end{itemize}
 If there is a   completely Lipschitz in small scale embedding   $f\colon B_{X}\to Y$, then $\theta= \gamma$.
\end{corollary}

Note that due to the very definition of $\kappa_n$, Theorems \ref{ThmLipPreservatinOfKappa},  \ref{ThmLowerBoundAlpha}, and \ref{ThmUpperBound} deal exclusively with maps $B_X \to Y$ where $X$ and $Y$ are homogeneous Hilbertian operator spaces. 
In Section \ref{SectionLocalized}, we use local techniques to push things beyond the Hilbertian setting and prove a result similar to the lower bound in Theorem \ref{ThmLipPreservatinOfKappa} in the case where $X$ is not Hilbertian (see Theorem \ref{ThmLowerBoundAlphaLocal}).
We defer the detailed statement to Section \ref{SectionLocalized} to avoid introducing various technical definitions here, and for the moment we only state a corollary of it to illustrate 
the kind of results we obtain.
Recall that from  Dvoretzky's theorem for operator spaces \cite{Pisier-Dvoretzky}, for any infinite-dimensional operator space $X$ there is an infinite-dimensional homogeneous Hilbertian operator space $Z$ which is completely isometric to a subspace of an ultrapower of $X$ (see \cite[Section 2.8]{Pisier-OS-book} for more details on ultraproducts of operator spaces). We call such a space a \emph{Dvoretzky space for $X$}.

\begin{corollary}\label{CorLowerBoundLipDvoretzky}
Let $X$ be an infinite-dimensional operator space, and let $Z$ be a Dvoretzky space for $X$.
Let $Y$ be a homogeneous Hilbertian operator space such that $\kappa_n(Y) \gtrsim n^{c}$ for some $c\in[0,1/2]$. 
If there is a  completely Lipschitz  in small scale embedding $B_{X}\to Y$, then $\kappa_n(Z) \gtrsim n^{c/(1+2c)}$.
\end{corollary}

The conclusion in Corollary \ref{CorLowerBoundLipDvoretzky} is weaker than the lower bound in Theorem \ref{ThmLipPreservatinOfKappa}, which is not surprising since the assumption on $Z$ is  weaker. Moreover, let us emphasize that it is significantly weaker: while the finite-dimensional subspaces of $Z$ are uniformly isomorphic to subspaces of $X$ because ultrapowers of a Banach space are finitely representable in the original space \cite[Proposition 6.1]{Heinrich}, in the operator space setting the corresponding statement with \emph{complete} isomorphisms does not hold \cite[Page 88]{Effros-Junge-Ruan}.

\section{Revisiting completely coarse maps}

In this section, we prove Theorem \ref{ThmCompleteCoarseAreRLinear}. The next lemma gives a  sufficient condition for  a map $B_X\to Y$ to be the restriction of an $\R$-linear map. 

\begin{lemma}\label{lemma-functional-equation}
Let $X$  and $Y$ be normed $\R$-vector spaces, and let $f :   B_X \to Y$ be a bounded  function such that $f(0)=0$  and

\[f\Big(\frac{1}{2}(x+z) \Big) = \frac{1}{2}\big( f(x)+f(z) \big)\] for all $x,z \in B_X$. Then $f$ is the restriction of an $\R$-linear function $X\to Y$.
\end{lemma}

\begin{proof}
Firstly, for computational reasons, it will be useful to assume that $f$ is defined on $2\cdot B_X$. This is not an issue since, replacing $f$ with $f(\frac{\cdot}{2})$, we can assume that $f$ is defined on the whole $2\cdot B_X$ and it still satisfies the assumptions of the lemma for all $x,z\in 2\cdot B_X$. Moreover, since $f(0)=0$, we must have 
\begin{equation}\label{Eq0LemmaLinear}
f\Big(\frac{1}{2}x\Big)=\frac{1}{2}f(x)\ \text{ for all }\ x\in 2\cdot  B_X.
\end{equation} 
Therefore, in order to show that $f\colon 2\cdot B_X\to Y $ is the restriction of an $\R$-linear function, it is enough that $f\restriction B_X$ is so.

\begin{claim}For all $x_1,\ldots, x_n\in B_X$, we have
\[f\Big(\frac{x_1+\ldots +x_n}{n} \Big) = \frac{1}{n}\Big( f(x_1)+\ldots+f(x_n) \Big).\]
\end{claim}

\begin{proof}
This follows from  induction on $n$. For $n=1$, the result is trivial; suppose then it holds for some $n\in\N$.  By  \eqref{Eq0LemmaLinear}, we have 
\begin{equation}\label{Eq1LemmaLinear}
f\Big(\frac{x+z}{2}\Big)=\frac{1}{2}f(x)+\frac{1}{2}f(z)=f\Big(\frac{x}{2}\Big)+f\Big(\frac{z}{2}\Big)
\end{equation}
for all $x,z\in 2\cdot B_X$.  Analogously, the induction hypothesis also implies that 
\begin{equation}\label{Eq2LemmaLinear}
f\Big(\frac{x_1+\ldots+x_n}{n}\Big)=f\Big(\frac{x_1}{n}\Big)+\ldots+f\Big(\frac{x_n}{n}\Big)
\end{equation}
for all $x_1,\ldots, x_n\in B_X$ (notice that the induction hypothesis does not allow us to conclude this holds for all elements in $2\cdot B_X$ though).

Notice that if  $x_1,\ldots, x_{n+1}\in B_X$, then \[\frac{2x_1}{n+1}\in B_X\ \text{ and }\ \frac{2(x_2+\ldots +x_{n+1})}{n+1}\in 2\cdot B_X.\] Therefore, by \eqref{Eq1LemmaLinear} and \eqref{Eq2LemmaLinear}, we have  
\begin{align}\label{Eq.25LemmaLinear}
f\Big(\frac{x_1+\ldots +x_{n+1}}{n+1} \Big)&=f\Big(\frac{2x_1/(n+1) +2(x_2+\ldots +x_{n+1})/(n+1)}{2} \Big)\\
&=f\Big(\frac{x_1}{n+1}\Big)+f\Big(\frac{x_2+\ldots +x_{n+1}}{n+1} \Big)\notag\\
&=f\Big(\frac{x_1}{n+1}\Big)+f\Big(\frac{n x_2/(n+1)+\ldots +n x_{n+1}/(n+1)}{n} \Big)\notag\\
&=f\Big(\frac{x_1}{n+1}\Big)+f\Big(\frac{x_2}{n+1}\Big)+\ldots+f\Big(\frac{x_{n+1}}{n+1}\Big)\notag
\end{align}
for all $x_1,\ldots, x_{n+1}\in B_X$.  In particular, if all $x_1,\ldots, x_{n+1}$ are the same, say   $x=x_i$ for all $i\in \{1,\ldots,n\}$, this shows that
\[f\Big(\frac{x}{n+1}\Big)=\frac{1}{n+1}f(x) \ \text{ for all }\ x\in B_X.\]
Revisiting \eqref{Eq.25LemmaLinear} with this extra information, we  conclude that 
 \begin{align*}
f\Big(\frac{x_1+\ldots +x_{n+1}}{n+1} \Big)=\frac{1}{n+1}\Big(f(x_1)+\ldots+f(x_{n+1})\Big)
\end{align*}
for all $x_1,\ldots, x_{n+1}\in B_X$ as desired. 
\end{proof}

The previous claim together with the fact that $f(0)=0$ gives 
\begin{equation}\label{Eq3LemmaLinear}
f(qx)=qf(x)\ \text{ for all }q\in [0,1]\cap \Q\ \text{ and all } x\in B_X,
\end{equation}
which also implies that  
\begin{equation}\label{Eq4LemmaLinear}
f(qx)=qf(x)\ \text{ for all }q\in [0,\infty)\cap \Q\ \text{ and all } x\in B_X\ \text{ with }qx\in B_X.
\end{equation}
For each $x\in X\setminus \{0\}$, pick  $r_x\in [2,3)$ such that $\|r_xx\|_X\in \Q$. Define a map $F\colon X\to Y$ by letting 
\begin{equation}
F(x)=\left\{\begin{array}{ll}
r_x\|x\|_Xf\Big(\frac{x}{r_x\|x\|_X}\Big), & x\neq 0,\\
0,& x=0.
\end{array}\right.
\end{equation}
It follows immediately from \eqref{Eq4LemmaLinear}
 that $F$ is an extension of $f\restriction B_X$. We are left to notice that $F$ is $\R$-linear. For additivity, notice that if $x,z\in\tfrac{1}{2}\cdot B_X$, then $x+z\in B_X$ and, by \eqref{Eq4LemmaLinear}, we must have 
 \begin{equation}\label{Eq5LemmaLinear}
 f(x+z)=f\Big(2\frac{x+z}{2}\Big)=2f\Big(\frac{x+z}{2}\Big)=f(x)+f(z).
 \end{equation}
 Fix  $x,z\in X$ and pick $M>1$   large enough so that 
 \[\frac{x}{M r_{x+z}\|x+z\|_X},\frac{z}{M r_{x+z}\|x+z\|_X}\in \frac{1}{2}\cdot B_X\] and 
 \[\frac{r_x\|x\|)X}{M r_{x+z}\|x+z\|_X},\frac{r_z\|z\|_X}{M r_{x+z}\|x+z\|_X}\leq 1.\]
 Then,  \eqref{Eq3LemmaLinear} and \eqref{Eq5LemmaLinear} together imply   
 \begin{align*}
\frac{1}{M} \cdot F(x+z)&=\frac{r_{x+z}\|x+z\|_X}{M}f\Big(\frac{x+z}{r_{x+z}\|x+z\|_X}\Big)\\
&=r_{x+z}\|x+z\|_Xf\Big(\frac{x+z}{M r_{x+z}\|x+z\|_X}\Big)\\
 &=r_{x+z}\|x+z\|_X\Big(f\Big(\frac{x}{M r_{x+z}\|x+z\|_X}\Big)+f\Big(\frac{z}{M r_{x+z}\|x+z\|_X}\Big)\Big)\\
 &=\frac{1}{M}(F(x)+F(z)).
 \end{align*}
 So, $F$ is additive.

We now show that $F(tx)=tF(x)$ for all $t\geq 0$ and $x\in X$. For that, fix $x\in X$ and define a map $g_x\colon \R\to X$ by letting \[g_x=F(tx)-t F(x)\ \text{ for all }\ t\in \R.\]
As $f$ is bounded,  $g_x$ is bounded on bounded sets. On the other hand,  since $F$ is additive, we have 
\[g_x(t+1)=F(tx+x)-(t+1) F(x)=F(tx)-t F(x)=g(t)\] for all $ t\in \R,$
i.e.,   $g_x$ is $1$-periodic. Therefore, $g_x$ must be bounded. However, as $F$ is additive, so is $g_x$. Since a bounded additive function must be zero, the result follows. \end{proof}

 The following is elementary and it is the operator space version of \cite[Lemma 	1.4]{KaltonNonlinear2008}. This will be used in the proof of Theorem \ref{ThmCompleteCoarseAreRLinear} below.

\begin{proposition}\label{Prop1}
Let $X$ and $Y$  be operator spaces and $K\subseteq X$ be convex. Then a map  $f : K \to Y$  is   completely
coarse if and only if  there is $C>0$ so that 
\[\|f([x_{ij}])-f([y_{ij}])\|_{\M_n(Y)}\leq C\|[x_{ij}]-[y_{ij}]\|_{\M_n(X)}+C\]
for all $n\in\N$ and all $[x_{ij}]\in \M_n(K)$.   \qed
\end{proposition}

Before presenting the proof of Theorem \ref{ThmCompleteCoarseAreRLinear}, we recall the concept of Hadamard matrices: a \emph{Hadamard matrix} is a square
matrix whose entries are either $1$ or $-1$ and whose rows are mutually
orthogonal. To notice that Hadamard matrices of arbitrarily large size exist, we define matrices $A_{2^k}\in \M_{2^k}(\C)$ inductively by letting
\[A_2=\begin{pmatrix}
    1& 1\\
    1& -1
\end{pmatrix}\ \text{ and }\ A_{2^{k+1}}=\begin{pmatrix}
    A_{2^{k}}& A_{2^{k}}\\
    A_{2^{k}}& -A_{2^{k}}
\end{pmatrix}\]
for all $k>1$.

\begin{proof}[Proof of Theorem \ref{ThmCompleteCoarseAreRLinear}]
Our goal to conclude that $f\colon B_X\to Y$ is the restriction of an $\R$-linear map is to use Lemma  \ref{lemma-functional-equation}. Hence, we fix distinct $x,z\in B_X$ and show that 
\[f\Big(\frac{1}{2}(x+z) \Big) = \frac{1}{2}\big( f(x)+f(z) \big).\]
For convenience, let \[x_0=\frac{x+z}{2} \text{ and } \ h=\frac{x-z}{2}.\]
so   $x_0$,   $h$, $x_0+h$, and $x_0-h$ are still in $B_X$, and we are left to show that 
\begin{equation}\label{Eq1Thm1}
f(x_0) =  \frac{1}{2}\Big( f(x_0+h)+f(x_0-h) \Big).
\end{equation}
As $x\neq z$, $h\neq 0$ and we can use   Hahn-Banach to pick $\varphi\in X^*$ with $\varphi(h)=1$. We then set 
\[y_0 = \frac{f(x_0-h)-f(x_0+h)}{2}\]
and define a map $g:X\to Y$ by letting \[g(x)= f(x) + \varphi(x)y_0\ \text{ for all } \ x\in B_X.\]
Since $f$ and $\varphi$ are completely coarse, so is $g$.
By Proposition \ref{Prop1}, there exists a constant $C>0$ such that for any $n\in\N$ and  $[x_{ij}]_{ij},[y_{ij}]_{ij} \in \M_n(B_X)$, we have
\begin{equation}\label{eqn-coarse-Lip}
    \n{ \big[ g(x_{ij})-g(y_{ij}) \big]_{ij} }_{\M_n(Y)} \le C \n{ [x_{ij}-y_{ij}]_{ij}}_{\M_n(X)} + C.
\end{equation}

Let  $ (A_{2^k} )_{k=1}^\infty$ be the Hadamard matrices and, for each $k\in\N$, write $A_{2^k} =  [ a^k_{i,j}  ]_{i,j=1}^{2^k}$. Since each $a^k_{i.j}$ is either $1$ or $-1$, we have that each $x_0 + ha^k_{ij}$ belongs to $B_X$. In particular, if $\mathbbm{1}_{2^k}$ denotes the $2^k\times 2^k$ scalar matrix whose entries are all 1,
\[\mathbbm{1}_{2^k}\otimes x_0 +  A_{2^k} \otimes h =[x_0 + ha^k_{ij}]_{i,j=1}^{2^k}\in \M_{2^k}(B_X)\]
and 
 \[\mathbbm{1}_{2^k}\otimes x_0  =[x_0]_{i,j=1}^{2^k}\in \M_{2^k}(B_X)\]
 for all $k\in\N$.
Hence, by \eqref{eqn-coarse-Lip},
\[
\n{ \big[ g(x_0 + ha^k_{ij})-g(x_0) \big]_{i,j=1}^{2^k} }_{\M_{2^k}(Y)} \le C \n{ A_{2^k} \otimes h }_{\M_{2^k}(X)} + C.
\]
Notice that, by the formula of $g$,    $g(x_0 +h) = g(x_0 -h)$. Therefore 
\[
\n{g(x_0+h)-g(x_0)}_Y \cdot\n{ \mathbbm{1}_{2^k} }_{\M_{2^k}} \le C \cdot \n{h}_X \cdot \n{ A_{2^k} }_{\M_{2^k}} + C,
\]
which yields
\[
\n{g(x_0+h)-g(x_0)}_Y2^{k} \le C \n{h}_X \sqrt{2^k} + C.
\]
Letting $k$ tend to infinity, we can conclude that  $g(x_0+h)=g(x_0)$. Unfolding definitions, this means that 
\[f(x_0)+\varphi(x_0)y_0 = f(x_0+h) + \varphi(x_0+h)y_0,\]
which, rearranging the terms and using that $\varphi(h)=1$, give $f(x_0) = f(x_0+h) +y_0$. By the definition of $y_0$, this implies
\[
f(x_0) =  f(x_0+h) + \frac{f(x_0-h)-f(x_0+h)}{2}= \frac{1}{2}\Big( f(x_0+h)+f(x_0-h) \Big).
\] This shows that \eqref{Eq1Thm1} holds and we are done. 
\end{proof}

\begin{remark}
The results of \cite{BragaChavezDominguez2020PAMS} are stated for both real and complex operator spaces. Since we make heavy use of complex interpolation in the present work, for simplicity we have written the whole paper in terms of complex operator spaces. Nevertheless, note that the proof of Theorem \ref{ThmCompleteCoarseAreRLinear} yields the same result for real operator spaces.     
\end{remark}

\section{Nonlinear small scale maps}\label{SectionNonlinearSmallScaleMaps}
 In this section, we produce examples of maps $B_X\to Y$ which are completely Lipschitz in small scale but are not the restriction of $\R$-linear maps. This culminates in Theorem  
\ref{ThmCompLipOnSmallScaleExistPol} and, more generally, in Theorem \ref{ThmConstructCompLipSmallScale}.

We start recalling some standard notation. Given $m\in\N$ and Banach   spaces $X_1,\ldots, X_m$, the sum $\bigoplus_{k=1}^mX_k$ 
 will always denote the Banach space obtained by  considering the vector space $\bigoplus_{k=1}^mX_k$  endowed with the $\ell_\infty$-sum. If moreover each $X_k$ is an operator space, then $\bigoplus_{k=1}^mX_k$  is also an operator space endowed with the $\ell_\infty$-sum operator space structure.  Since $m$-linear maps will play an important role in what follows, we quickly recall some basic terminology. Given Banach spaces $X_1,\ldots, X_m,Y$, the \emph{norm of an $m$-linear map} $Q\colon 
\bigoplus_{k=1}^mX_k  \to Y$ is the infimum  of all $L>0$ such that 
\[\left\|Q(x^{(1)},\ldots, x^{(m)})\right\|_Y\leq L\|x^{(1)}\|_{X_1}\cdot \ldots \cdot\|x^{(m)}\|_{X_m}\]
for all $(x^{(1)},\ldots, x^{(m)})\in \bigoplus_{k=1}^mX_k$. This infimum is denoted by $\|Q\|$ and we say that $Q$ is a \emph{bounded $m$-linear map} if  $\|Q\|<\infty$. Notice that, if $X_1,\ldots, X_m, Y$ are moreover operator spaces, then each $n$-amplification $Q_n\colon 
\bigoplus_{k=1}^m\M_n(X_k)  \to \M_n(Y)$ is also $m$-linear and hence  $\|Q_n\|$ is well defined.  
 
 \begin{definition}
 Let $m\in\N$, $X_1,\ldots,X_m, Y$ be operator spaces, and $Q\colon \bigoplus_{k=1}^mX_k\to Y$ be an $m$-linear map. We say that $Q$ is \emph{completely controlled} if\footnote{We chose to control (no pun intended) the automatic instinct of calling such $m$-linear map \emph{completely bounded} since this definition already exists for $m$-linear maps and it is not the one given above (see, for instance, \cite{Christensen}). }
\[\|Q\|_{\mathrm{cc}}=\sup_n\|Q_n\|<\infty. \] 
 \end{definition}

\begin{proposition}\label{PropExamplesCompContMaps}
Let $m\in\N$, $H$ be a Hilbert space, and $X_1,\ldots,X_m\subseteq \cB(H)$ be operator spaces. The product map $P\colon \bigoplus_{k=1}^mX_k\to \cB(H)$ given by \[P(x^{(1)},\ldots, x^{(m)})=x^{(1)}\cdot\ldots\cdot x^{(m)}\ \text{ for all }\ (x^{(1)},\ldots, x^{(m)})\in \bigoplus_{k=1}^mX_k\]
 is a completely controlled $m$-linear map with $\|P\|_{\mathrm{cc}}\leq 1$.
\end{proposition}

\begin{proof}
The $m$-linearity of $P$ is straightforward. To notice that $P$ is completely controlled,
the crucial tool   is the following generalization of Schur's inequality \cite[Satz III]{Schur} which follows from \cite[Theorem 2.3]{Christensen}: if $n\in\N$ and $[x_{ij}],[z_{ij}]\in \M_n(\mathcal{B}(H))$, then
\begin{equation*}\label{eqn-Schurs-inequality-B(H)}
\|[x_{ij}z_{ij}]\|_{\M_n(\mathcal{B}(H))}\leq \|[x_{ij}]\|_{\M_n(\mathcal{B}(H))}\|[z_{ij}]\|_{\M_n(\mathcal{B}(H))}.    
\end{equation*}
By a straightforward induction, this implies that 
\begin{align*}\label{eqn-Schurs-inequality-B(H)}
\left\|P_n([x^{(1)}_{ij}], \ldots ,[ x^{(m)}_{ij}])\right\|_{\M_n(\mathcal{B}(H))}&=\left\|[x^{(1)}_{ij}\cdot \ldots \cdot x^{(m)}_{ij}]\right\|_{\M_n(\mathcal{B}(H))}\\
&\leq \left\|[x^{(1)}_{ij}]\right\|_{\M_n(X_1)}\cdot\ldots\cdot\left\|[x^{(m)}_{ij}]\right\|_{\M_n(X_m)}   
\end{align*}
for all $n\in\N$ and all  $([x^{(1)}_{ij}],\ldots, [x^{(m)}_{ij}])\in \bigoplus_{k=1}^m \M_n(X_k)$.
 This shows that $\|P\|_{\mathrm{cc}}\leq 1$.
\end{proof}

The following is our main result to construct nontrivial examples of non-$\R$-linear maps which are completely Lipschitz in small scale. 
 
 \begin{theorem}\label{ThmConstructCompLipSmallScale}
 Let $m\in\N$, $X,X_1,\ldots,X_m, Y$ be operator spaces, $T\colon X\to \bigoplus_{k=1}^mX_k$ be a completely bounded operator, and $Q\colon \bigoplus_{k=1}^mX_k\to Y$ be a completely controlled   $m$-linear map. Then $Q\circ T\colon B_X\to Y$ is completely Lipschitz in small scale.
 \end{theorem}

We start with a lemma about small scale behavior of $m$-linear maps on Banach spaces.

\begin{lemma}\label{LemmamLinearMapsBanachSpacesLip}
Let $m\in\N$, $X_1,\ldots,X_m,Y$ be Banach spaces, and $ Q\colon \bigoplus_{k=1}^mX_k\to  Y$ be a bounded  $m$-linear map. Then, for all  $(x^{(1)},\ldots, x^{(m)}),(z^{(1)},\ldots, z^{(m)} )\in B_{ \bigoplus_{k=1}^mX_k}$, we have 
\begin{align*}
\left\|Q(x^{(1)},\ldots, x^{(m)})\right. &-\left.Q(z^{(1)},\ldots, z^{(m)})
\right\|_{Y}\\
&\leq {m}\|Q\|\left\|(x^{(1)},\ldots, x^{(m)})-(z^{(1)},\ldots, z^{(m)})\right\|_{\bigoplus_{k=1}^mX_k}.
\end{align*}
 \end{lemma}
 
 \begin{proof}
We proceed by induction on $m\in\N$. If $m=1$, $Q$ is a bounded linear operator and the result is immediate. Suppose it holds for $m-1$, with $m\geq 2$, and let us show it is also valid for $m$.  Fix  $x^{(1)},\ldots, x^{(m)},z^{(1)},\ldots, z^{(m)} $ in  $B_{\bigoplus_{k=1}^{m}X_k}$. Then, as $Q$ is $m$-linear and as each $x^{(k)}$ has norm at most $1$, it follows that 

\begin{align*}
\left\|Q(x^{(1)},\ldots, x^{(m-1)}, x^{(m)})\right. - &\left.Q(x^{(1)},\ldots, x^{(m-1)},  z^{(m)})
\right\|_{Y}\\
&=\left\|Q(x^{(1)},\ldots, x^{(m-1)},  x^{(m)}- z^{(m)})
\right\|_{Y}\\
&\leq  \|Q\|\left\|x^{(m)}-z^{(m)}\right\|_{X_{m}}.
\end{align*}

Let $ B\colon \bigoplus_{k=1}^{m-1}X_k\to  Y$ be given by 
\[B(w^{(1)}, \ldots, w^{(m-1)})=Q(w^{(1)},\ldots, w^{(m-1)}, z^{(m)})\]
for all $(w^{(1)},\ldots, w^{(m-1)})\in \bigoplus_{k=1}^{m-1}X_k$. So, $B$ is an $(m-1)$-linear map and, as $\|z^{(m)}\|_{X_m}\leq 1$, we have   $\|B\|\leq \|Q\|$. Therefore,  the induction hypothesis gives that 
\begin{align*}
\left\|B(x^{(1)},\ldots,\right. & x^{(m-1)}) - \left.B(z^{(1)},\ldots, z^{(m-1)})
\right\|_{Y}\\
&\le(m-1)\|Q\|\left\|(x^{(1)},\ldots, x^{(m-1)})-(z^{(1)},\ldots, z^{(m-1)})\right\|_{\bigoplus_{k=1}^{m-1}X_k}.
\end{align*}
The lemma then follows by the triangle inequality. 
 \end{proof}

 \begin{proof}[Proof of Theorem  \ref{ThmConstructCompLipSmallScale}]
Since each amplification $Q_n$ is an $m$-linear map with norm at most $\|Q\|_{\mathrm{cc}}$,  Lemma \ref{LemmamLinearMapsBanachSpacesLip} gives that
\begin{align*}
\left\|Q_n([x^{(1)}_{ij}],\ldots,\right. & [x^{(m)}_{ij}]) -\left. Q_n([z^{(1)}_{ij}],\ldots,[ z^{(m)}_{ij}])
\right\|_{\M_n(Y)}\\
&\leq {m}\|Q\|_{\mathrm{cc}}\left\|([x^{(1)}_{ij}],\ldots, [x^{(m)}_{ij}])-([z^{(1)}_{ij}],\ldots, [z^{(m)}_{ij}])\right\|_{\bigoplus_{k=1}^m\M_n(X_k)} 
 \end{align*}
  for all $n\in\N$ and all 
 $([x^{(1)}_{ij}],\ldots, [x^{(m)}_{ij}]),([z^{(1)}_{ij}],\ldots, [z^{(m)}_{ij}] )\in B_{ \bigoplus_{k=1}^m\M_n(X_k)}$. By the $m$-linearity of $Q_n$, if $([x^{(1)}_{ij}],\ldots, [x^{(m)}_{ij}])$ and $([z^{(1)}_{ij}],\ldots, [z^{(m)}_{ij}] )$ have norm at most $\|T\|_{\cb}$ instead, we obtain the a similar inequality with the factor $\|T\|_{\cb}^m$ added to it on the right-hand side.  Therefore,  we conclude that 
  \begin{align*}
  \left\|(Q\circ T)_n([x_{ij}])\right.- &\left. (Q\circ T)_n([z_{ij}])	\right\|_{\M_n(Y)}\\
  &=
\left\|Q_n(T_n([x_{ij}]))- Q_n(T_n([z_{ij}]))\right\|_{\M_n(Y)}\\
&\leq m\|Q\|_{\mathrm{cc}}\|T\|_{\mathrm{cb}}^m  \left\|[T(x_{ij})]-[T(z_{ij})]\right\|_{\bigoplus_{k=1}^m\M_n( X_k)}\\
&\leq  m\|Q\|_{\mathrm{cc}}\|T\|_{\mathrm{cb}}^{m+1}  \left\|[x_{ij}]-[z_{ij}]\right\|_{\M_n( X)}
 \end{align*}
  for all $n\in\N$ and all 
 $[x_{ij}],[z_{ij}]\in B_{\M_n(X)}$. So, \[\omega_{Q\circ T}^{\mathrm{ss}}(t)\leq  m\|Q\|_{\mathrm{cc}}\|T\|_{\mathrm{cb}}^{m+1} t\] for all $t\geq 0$.
 \end{proof}

 \begin{proof}[Proof of Theorem \ref{ThmCompLipOnSmallScaleExistPol}]
Let $m\in\N$ and let $p^m$ be the complex polynomial given by  $p^m(x)=x^m$ for all $x\in \C$. Let $A$ be an operator algebra. Then, by  Proposition \ref{PropExamplesCompContMaps}, the map $(a_1,\ldots,a_m)\in \bigoplus_{k=1}^m A\mapsto a_1\cdots a_m\in A$ is completely controlled. Since the map $a\in A\mapsto (a,\ldots,a)\in \bigoplus_{k=1}^m A$ is completely bounded, it follows from   Theorem  \ref{ThmConstructCompLipSmallScale} that $p^m$ is  completely Lipschitz in small scale.
As any polynomial is a linear combination of polynomials of the form $p^m$, $m\in\N$, the result follows.  \end{proof}

The reader familiar with $m$-linear maps on operator spaces knows that other notions of ``complete boundedness'' are also frequently studied.  More precisely,  if $Q\colon \bigoplus_{k=1}^mX_k\to Y$ is an $m$-linear map between operator spaces, $Q$ may be \emph{completely bounded} or \emph{jointly completely bounded}. Those are technical definitions which, for brevity, we chose to omit  here, see   \cite[Page 281]{ChristensenEffrosSinclair1987Inventiones}  and \cite[Section  1.5.11]{BlecherLeMerdy2004} for the precise definitions of each of them, respectively.\footnote{
The reader should be warned that some references such as \cite{Effros-Ruan-book} use the terminology completely bounded (resp. multiplicatively bounded) for the multilinear mappings that are nowadays generally called jointly completely bounded (resp. completely bounded).
}  We finish this section briefly relating these notions.

\begin{proposition}\label{PropJCBisCompCont}
Let $m\in\N$, $X_1,\ldots,X_m, Y$ be operator spaces, and $Q \colon \bigoplus_{k=1}^mX_k\to Y$ be a jointly completely bounded $m$-linear map. Then $Q$ is completely controlled with $\|Q\|_{\mathrm{cc}} \leq \|Q\|_{\mathrm{jcb}}$, where $\|Q\|_{\mathrm{jcb}}$ denotes the norm of joint complete boundedness of $Q$ (see \cite[Section  1.5.11]{BlecherLeMerdy2004}). 
\end{proposition}

\begin{proof}
For simplicity we will write the proof in the case $m=2$, the general case is analogous.
Let $n\in\N$.
Since $Q$ is jointly completely bounded, for every $[x^{(1)}_{ij}] \in \M_n(X_1)$ and $[x^{(2)}_{ij}] \in \M_n(X_2)$ we have
\[
\n{[Q(x^{(1)}_{ij},x^{(2)}_{kl})]}_{\M_{n^2}(Y)} \le \|Q\|_{\mathrm{jcb}} \left\|[x^{(1)}_{ij}]\right\|_{\M_n(X_1)} \left\|[x^{(2)}_{ij}]\right\|_{\M_n(X_2)}
\]
The desired conclusion now follows from Ruan's axioms by observing that $[Q(x^{(1)}_{ij},x^{(2)}_{ij})]$ is an $n\times n$ principal submatrix of the $n^2\times n^2$ matrix $[Q(x^{(1)}_{ij},x^{(2)}_{kl})]$, that is, there exists a coordinate partial isometry $P \in \M_{n,n^2}(\C)$ such that $[Q(x^{(1)}_{ij},x^{(2)}_{ij})] = P[Q(x^{(1)}_{ij},x^{(2)}_{kl})]P^*$.
\end{proof}

Since completely bounded $m$-linear maps are jointly completely bounded, Proposition \ref{PropJCBisCompCont} shows that many of the multilinear maps that have previously been studied in operator space theory are completely controlled.
Also, in the case of maximal operator spaces there is an abundance of completely controlled maps (see \cite[Chapter 3]{Pisier-OS-book} for the definition of the minimal and maximal operator space structures).

\begin{corollary}\label{CorCompContMaximal}
Let $m\in\N$, $X_1,\ldots,X_m$ be Banach spaces, and $Y$ be an operator space.
Then any bounded $m$-linear map $Q\colon \bigoplus_{k=1}^m X_k\to Y$ is completely controlled as a map $\bigoplus_{k=1}^m \max(X_k)\to Y$.    
\end{corollary}

\begin{proof}
By Proposition \ref{PropJCBisCompCont}, it suffices to show that $Q$ is jointly completely bounded.
Since jointly completely bounded $m$-linear maps correspond to completely bounded linear maps on the projective operator space tensor product (see \cite[Proposition  7.1.2]{Effros-Ruan-book}), and
$\max(X_1) \widehat\otimes \cdots \widehat\otimes \max(X_m) = \max( X_1 \widehat\otimes_\pi \cdots \widehat\otimes_\pi X_m)$, where $\otimes_\pi$ is the Banach space projective tensor product (see \cite[Proposition 1.5.12]{BlecherLeMerdy2004}), the conclusion follows from the fact that bounded linear maps whose domain is a maximal operator space are automatically completely bounded.
\end{proof}

The reader familiar with the theory of polynomials on vector spaces will already have recognized that in the situation of Theorem \ref{ThmConstructCompLipSmallScale}, if $X=X_1=\cdots=X_m$ and $T$ is the diagonal map $T(x) = (x,x,\dotsc,x)$, the composition $Q \circ T$ is precisely a polynomial on $X$. Therefore, one would expect that the available literature on polynomials on operator spaces might already provide us with more examples of maps which are completely Lipschitz in small scale. However, this is a subject that has not yet been developed much: the only significant works in this regard appear to be \cite{Dineen-Radu-polynomials,Defant-Wiesner}.
Corollary \ref{CorCompContMaximal} is closely related to \cite[Proposition 9.3]{Defant-Wiesner}, and we next list other examples of completely Lipschitz in small scale maps that follow from the aforementioned two papers.

\begin{corollary}
Let $m\in\N$ and $Y$ an operator space.
Then any bounded $m$-linear map $Q\colon \bigoplus_{k=1}^m \ell_\infty\to Y$ is completely controlled.    
\end{corollary}

\begin{proof}
This follows from \cite[Proposition 9.5]{Defant-Wiesner}, since the Schur multilinear mappings defined in that paper are easily seen to be completely controlled and thus Theorem \ref{ThmConstructCompLipSmallScale} yields the desired result.     
\end{proof}

For our last example, the operator space $\mathrm{OH}$ mentioned in it is the Hilbert operator space of G. Pisier, see \cite[Chapter 7]{Pisier-OS-book} for its definition.

\begin{corollary}
Let $(e_j)_{j=1}^\infty$ be an orthonormal basis for $\mathrm{OH}$, $Y$ an operator space, and $(y_j)_{j=1}^\infty$ a norm null sequence in $Y$. Let $m \ge 2$ be a natural number and define $P : \mathrm{OH} \to Y$ by
\[
P\left( \sum_{n=1}^\infty x_j e_j \right) = \sum_{n=1}^\infty x_j^m y_j.
\]
Then $P$ is completely Lipschitz in small scale.
\end{corollary}

\begin{proof}
This follows  from \cite[Proposition  4.1]{Dineen-Radu-polynomials}, since the completely bounded polynomials defined in that paper are constructed as restrictions to the diagonal of jointly completely bounded multilinear maps.   
Thus, Proposition \ref{PropJCBisCompCont} and Theorem \ref{ThmConstructCompLipSmallScale} yield the desired result.
\end{proof}

\section{Basics about  $\kappa_n(X)$ and examples}\label{SectionComputeKappa}

 In this section, we prove basic properties about the $\kappa_n$'s and compute $\kappa_n(X)$ for many homogeneous Hilbertian operator spaces $X$. For its definition, see Section \ref{SectionIntro}. 
 
 We start recalling   the definition of the row and column operator spaces. Given $i,j\in\N$, we let $e_{i,j}$ denote the operator on $\ell_2$ whose matrix representation has $1$ in the $(i,j)$-th entry and zero elsewhere. The \emph{row} and the \emph{column operator spaces} are then defined to be 
\[R=\overline{\mathrm{span}}\{e_{1,j}\mid j\in\N\}\ \text{ and }\ C=\overline{\mathrm{span}}\{e_{i,1}\mid i\in\N\}.\]
In particular, both $R$ and $C$ are homogeneous Hilbertian operator spaces.

\begin{example}\label{ExampleRC}
If $R$ and $C$ are the row and column operator spaces, respectively, then $\kappa_n(R)=1$ and $\kappa_n(C)=\sqrt{n}$ for all $n\in\N$. Moreover, it is clear that these spaces   minimize and maximize $\kappa_n$, respectively, i.e., for any operator space $E$ we have  $1\leq \kappa_n(E)\leq \sqrt{n}$ for all $n\in\N$. The lower bound is obvious since each of the canonical projections $\M_n(E)\to E$ is completely contractive and the upper bound follows equally as easily by considering some representation $E\subseteq \cB(H)$ and computing $b(x)$, where $b$ is the $n$-by-$n$ $E$-valued matrix in the definition of $\kappa_n(E)$ and $x$ is an arbitrary normalized vector in $H^{\oplus n}$.
\end{example}

\begin{proposition}\label{PropKappaBasics0}
For any  homogeneous Hilbertian operator space $X$ and any $n\in\N$, we have   $\kappa_n(X)\kappa_n(X^*)\geq \sqrt{n}$.
\end{proposition}

\begin{proof}
 Let $\{e_j\}_{j=1}^n$ be an arbitrary orthonormal set in $X$  and $\{e_k^*\}_{k=1}^n$ be an orthonormal set in $X^*$ which is biorthogonal to $\{e_j\}_{j=1}^n$. Then the product $\kappa_n(X) \kappa_n(X^*)$ dominates the norm in $\M_{n^2}$ of the matrix pairing between
\[
\begin{bmatrix}
e_1&0&\ldots &0\\
e_2&0&\ldots &0\\ 
\vdots& \vdots &\ddots&\vdots\\
e_n&0&\ldots &0
 \end{bmatrix} \quad \text{and} \quad \begin{bmatrix}
e^*_1&0&\ldots &0\\
e^*_2&0&\ldots &0\\ 
\vdots& \vdots &\ddots&\vdots\\
e^*_n&0&\ldots &0
 \end{bmatrix},
\]
which is exactly $\sqrt{n}$ since, as a matrix in $\M_{n^2}$, it has $n$ entries equal to $1$ in the first column and all other entries $0$.
\end{proof}

Interpolation spaces will provide  a good source of examples of operator spaces whose $\kappa_n$'s can be estimated.   For simplicity, we now set some notation:  if $X_0$ and $X_1$ are homogeneous Hilbertian operator spaces and $\theta\in [0,1]$, we let 
\[  X_\theta=(X_0,X_1)_\theta,\]
where the interpolation above is taken with respect to  some isometric identification $X_0\simeq X_1$ (since both spaces are homogeneous Hilbertian, the specific identification is irrelevant).  Due to its technical definition, we refer the reader to \cite[Section 2.7]{Pisier-OS-book} for the definition of interpolation operator spaces.

\begin{proposition}\label{PropKappaBasics}
Consider the interpolation spaces $(X_\theta)_{\theta\in [0,1]}$ of a given   pair  $(X_0,X_1)$ of infinite dimensional homogeneous Hilbertian operator spaces. The following holds. 
\begin{enumerate}
\item\label{PropKappaBasicsItem1} $\kappa_n(X_\theta)\leq \kappa_{n}(X_0)^{1-\theta}\kappa_n(X_1)^\theta$ for all $\theta\in [0,1]$ and all $n\in\N$.
\item\label{PropKappaBasicsItem3} If $X_0^*\equiv X_1$, then $\kappa_n(X_\theta)\geq\sqrt{n}/( \kappa_{n}(X_1)^{1-\theta}\kappa_n(X_0)^\theta)$.
\end{enumerate}
\end{proposition}

\begin{proof}
\eqref{PropKappaBasicsItem1} It is standard in interpolation theory that  \[\|b\|_{\M_n(X_\theta)}\leq \|b\|^{1-\theta}_{\M_n(X_0)}\|b\|^{\theta}_{\M_n( X_1) } \] for any $b\in \M_n(X_\theta)$ (see  \cite[Section 1.2.30]{BlecherLeMerdy2004}). So, the inequality  is immediate.

\eqref{PropKappaBasicsItem3}  Since $X_0$ is reflexive and using that $X_0^*=X_1$,  we have that
\begin{equation}\label{Eq.DualOfInterpolatedSpace}
X_\theta^*=(X_0,X_1)^*_\theta\equiv(X_0^*,X_1^*)_\theta\equiv(X_1,X_0)_\theta\equiv(X_0,X_1)_{1-\theta}    
\end{equation}
(\cite[Theorem 2.7.4]{Pisier-OS-book}). 
It then follows from \eqref{PropKappaBasicsItem3}  that $\kappa_n(X_\theta^*)\leq \kappa_{n}(X_0)^{\theta}\kappa_n(X_1)^{1-\theta}$. Therefore,  since Proposition \ref{PropKappaBasics0} gives $\kappa_n(X_\theta)\kappa_n(X^*_\theta)\geq \sqrt{n}$, the result follows. 
\end{proof}

\begin{corollary}\label{CorollaryRCKappa}
The following holds for $\theta\in [0,1]$ and $n\in\N$. 
\begin{enumerate}
\item \label{CorollaryRCKappaItem1}
 $\kappa_n((R,C)_\theta)=n^{\theta/2}$.
\item\label{CorollaryRCKappaItem2}
  $\kappa_n((\min(\ell_2),\max(\ell_2))_\theta)=n^{\theta/2}$.
\end{enumerate}
\end{corollary}

\begin{proof}
\eqref{CorollaryRCKappaItem1}
 It is immediate that $\kappa_n(R)=1$ and $\kappa_n(C)=\sqrt{n}$ (Example \ref{ExampleRC}). Therefore, since $R^*\equiv C$ (\cite[Page 41]{Pisier-OS-book}), 
the upper and lower bounds given by Proposition \ref{PropKappaBasics} imply  that $\kappa_n((R,C)_\theta)=n^{\theta/2}$.

\eqref{CorollaryRCKappaItem2}
Since the identity $R\to \min(\ell_2)$ is completely contractive and $\kappa_n(R)=1$, this implies $\kappa_n(\min(\ell_2))=1$. Similarly, since the identity $\max(\ell_2)\to C$ is completely contractive, $\kappa_n(\max(\ell_2))\geq \sqrt{n}$ and equality must then hold as   $\kappa_n(\max(\ell_2))\leq \sqrt{n}$ (Example \ref{ExampleRC}).   Therefore,   as   $\min(\ell_2)^*=\max(\overline{\ell_2})$ (\cite[Page 72]{Pisier-OS-book}) and as $\max(\overline{\ell_2})$ is completely isometric to $\max(\ell_2)$, the result then follows from  Proposition \ref{PropKappaBasics} again. 
\end{proof}

As our next result shows,  Corollary \ref{CorollaryRCKappa}\eqref{CorollaryRCKappaItem2} can be considerably generalized. Recall that, if $X$ is a homogeneous Hilbertian operator space, then Riesz representation gives us a linear isometry $X\to \overline{X^*}$ and this allows us to construct the interpolation spaces $(X,\overline{X^*})_\theta$. In order to present a consequence of this generalization,   we recall the definitions of $R\cap C$ and $R+ C$. Let $r\colon \ell_2\to R$ and $c\colon \ell_2\to C$ be canonical isometries. The operator  space $R\cap C$ is the Banach space $\ell_2$ together with the operator space structure given by the isometric inclusion
\[x\in \ell_2\mapsto (r(x),c(x))\in R\oplus C\subseteq \mathcal B(\ell_2)\oplus \mathcal B(\ell_2).\]
 The operator space $R+C$ is the quotient $(R\oplus_1 C)/\Delta$, where $\Delta=\{(r(x),-c(x))\mid x\in \ell_2\}$  (\cite[Page 194]{Pisier-OS-book}). 
 
\begin{corollary}\label{CoroRcapCR+C}
Let $X$ be a homogeneous Hilbertian space such that the identity $R\to X$ has $\cb$-norm 1. Then,   $\kappa_n((X,\overline{X^*})_\theta)=n^{\theta/2}$ for all $\theta\in [0,1]$ and all $n\in\N$. In particular, 
 $\kappa_n((R\cap C,R+C)_\theta)=n^{\theta/2}$  for all $\theta\in [0,1]$ and all $n\in\N$.
\end{corollary}

\begin{proof}
Since the identity  $R\to X$ has $\cb$-norm 1, its adjoint  $X^*\to C$ also has $\cb$-norm $1$. So, the proof follows exactly as the one of Proposition \ref{CorollaryRCKappa}\eqref{CorollaryRCKappaItem2}. The last statement follows since the identity $R\to R\cap C$ is clearly completely contractive and $(R\cap C)^*\equiv R+C$ (\cite[Page 194]{Pisier-OS-book}). 
\end{proof}

For our next example, we   recall the definition of   Fermionic   operator spaces. Let $H$ be a Hilbert space and $(v_i)_{i\in I}$ be a family of operators in $\cB(H)$ such that \[v_iv_j+v_jv_i=0\ \text{ and }\ v_iv_j^*+v_j^*v_i=\delta_{i,j}\mathrm{Id}_H\] for all $i,j\in I$, where $(\delta_{i,j})_{i,j\in I}$ are the  Kronecker deltas. Then, the \emph{Fermionic operator space associated to $I$} is 
\[\Phi(I)=\overline{\mathrm{span}}\{v_i\mid i\in I\}.\]
It turns out the space above does not depend on $(v_i)_{i\in I}$ per se but only on $I$. We refer to 
 \cite[Theorem 9.3.1]{Pisier-OS-book} for a proof of that. For shortness, if  $I=\{1,\ldots, n\}$, we write $\Phi_n$ for $\Phi(\{1,\ldots, n\})$.

\begin{proposition}\label{PropFermion}
For any infinite set $I$, we have that $\kappa_n(\Phi(I)) \simeq \sqrt{n}$.
\end{proposition}

\begin{proof}
The upper bound $\kappa_n(\Phi)\leq \sqrt{n}$ is immediate since it holds for any homogeneous Hilbertian operator space (Example \ref{ExampleRC}). 
Let $C_n$ denote $\ell_2^n$ with the column operator space structure, i.e., $C_n=\mathrm{span}\{e_{1,i}\mid i\in \{1,\ldots, n\}\}$. It is shown in  \cite[Equation  (10.22)]{Pisier-OS-book}  that the identities  $\Phi_n \to C_n$ have uniformly bounded $\cb$-norms. Therefore,  $\kappa_n(\Phi(I)) \gtrsim \kappa_n(C)$. Since, $\kappa_n(C)=\sqrt{n}$, the result follows.
\end{proof}

\begin{remark}
The examples for which we have been able to calculate $\kappa_n(X)$ suggest the following two questions for an arbitrary homogeneous Hilbertian operator space $X$:
\begin{enumerate}
\item Is the inequality in Proposition \ref{PropKappaBasics0} always an equality? That is, do we have $\kappa_n(X)\kappa_n(X^*) = \sqrt{n}$ for all $n\in\N$?
\item Is $\kappa_n(X)$ always equivalent to a power of $n$? That is, does there exist a constant $c(X)$ such that $\kappa_n(X) \simeq n^{c(X)}$? 
\end{enumerate}
\end{remark}

\section{Small scale rigidity of $\kappa_n(E)$.}
In this section, we now show how the $\kappa_n$'s impact the  existence of small scale Lipschitz maps between operator spaces. The results herein will culminate in   Theorem \ref{ThmLipPreservatinOfKappa}. For that, in order to obtain the equivalence $\kappa_n(X)\simeq \kappa_n(Y)$, we prove the inequalities $\lesssim$
and $\gtrsim$ separately, see Theorems \ref{ThmLowerBoundAlpha} and \ref{ThmUpperBound}, respectively. We  emphasize this here since   each of these partial results have weaker hypotheses than Theorem  \ref{ThmLipPreservatinOfKappa}.

We start with a lemma relating the $\kappa_n$'s  with a similar quantity computed with respect to weakly null sequences instead of orthonormal sets.

\begin{lemma}\label{LemmaLowerBoundInterpolation}
Let $X$ be an infinite dimensional homogeneous Hilbertian operator space, $b\geq a>0$, and let $(x_m)_m$ be a  weakly null sequence in $X$ such that $\|x_m\|\in [a,b]$ for all $m\in\N$. Then, for all $\eps>0$, there is an infinite $\mathbb M\subseteq \N$ such that
\[ \left(a-\eps\right)\kappa_n(E)\leq  \left\|\begin{bmatrix}
x_{m_1}&0&\ldots& 0\\
x_{m_2}&0&\ldots& 0\\
\vdots&\vdots&\ddots &\vdots\\
x_{m_n}&0&\ldots & 0
\end{bmatrix}\right\|_{\M_n(X)}\leq \left(b+\eps\right) \kappa_n(X).
\]  
for all $m_1<\ldots< m_n\in \mathbb M$. 
\end{lemma}

\begin{proof}
Let $(e_n)_n$ be an orthonormal sequence in $X$ and $\eps>0$, without loss of generality, assume $\eps<a$. By going to a subsequence if necessary, a  standard gliding-hump argument allows us to assume that the assignment $x_m\mapsto e_m$ defines an isomorphism $T\colon \overline{\mathrm{span}}\{x_m\mid m\in\N\}\to \overline{\mathrm{span}}\{e_m\mid m\in\N\}$ such that $\|T\|\leq 1/(a-\eps)$ and $\|T^{-1}\|\leq b+\eps$. As $X$ is a homogeneous space, we must also have  $\|T\|_{\cb}\leq 1/(a-\eps)$ and $\|T^{-1}\|_{\cb}\leq  b+\eps$, so the result follows. \end{proof}

We say that a map $f\colon B_X\to Y$ is \emph{completely bounded in small scale} if there is $M>0$ such that 
\[\|f_n([x_{ij}])\|_{\M_n(X)}\leq M\ \text{ for all }\ n\in\N\ \text{ and all }\ [x_{ij}]\in B_{\M_n(X)}.\]

In order to obtain restrictions for the existence of certain maps, we must also demand the maps to satisfy some nontrivial lower estimates. The next definition is an ``operator space/small scale''  version of the the compression exponent
of a metric space $X$ into another space $Y$ introduced in  \cite{GuentnerKaminker2004}.
 
\begin{definition}
Let $X$ and $Y$ be operator spaces. 
 We denote by $\alpha^{\mathrm{ss}}_Y(X)$ the infimum of all $\alpha\geq 1$
for which there is completely bounded in  small scale $f\colon B_X\to Y$ and $L\geq 1$ such that  
\[\|f(x)-f(y)\|_Y\geq \frac{1}{L}\|x-y\|^\alpha_X\]
for all $x,y\in B_X$.\footnote{We point out that  the compression exponent
of a metric space $X$ into another space $Y$ (see \cite{GuentnerKaminker2004}) considers the supremum of all  $\alpha\leq 1$ for which a similar inequality holds. This difference comes from the fact that the compression exponent deals with large scale geometry, while the exponent $\alpha^{\mathrm{ss}}_Y(X)$ is supposed to capture small scale behavior. }
\end{definition}

\begin{theorem}\label{ThmLowerBoundAlpha}
   Given arbitrary homogeneous Hilbertian operator spaces $X$ and $Y$, we have that $\kappa_n(X)^{\alpha^{\mathrm{ss}}_Y(X)}\gtrsim \kappa_n(Y)$.
\end{theorem}

 \begin{proof}
Let    $f\colon  B_{X}\to Y$ be  completely bounded in small scale. Suppose  $\alpha,L\geq 1$  are such that 
\begin{equation}\label{Eq1.Thm.alpha.RC}
\|f(x)-f(y)\|_Y\geq \frac{1}{L}\|x-y\|_X^\alpha\ \text{ for all }\ x,y\in B_X.
\end{equation}
  As $f$ is  completely bounded in small scale, fix  $M>0$  such that
\begin{equation}\label{Eq2.Thm.alpha.RC}  
  \|f_n(a)\|\leq M\ \text{  for all }\ a\in B_{\M_n(X)}.
\end{equation}

Fix $n\in\N$ and let $(e_j)_j$ be an orthonormal sequence in $X$. Appealing to  Rosenthal's $\ell_1$-theorem (see \cite[The Main Theorem]{Rosenthal1974PNAS}), by going to a subsequence if necessary, we can assume $(f(e_j/\kappa_n(X)))_j$ is weakly Cauchy. In particular,  $(f(e_{2j-1}/\kappa_n(X))-f(e_{2j}/\kappa_n(X)))_j$ is weakly null. Moreover, \eqref{Eq1.Thm.alpha.RC} implies that 
  \begin{align*}
\left\| 
f\left( \frac{e_{2j-1}}{\kappa_n(X)}\right)-f\left(\frac{e_{2j}}{\kappa_n(X)}\right) \right\|_{Y}  
  \geq \frac{1}{L\kappa_n(X)^\alpha}
\end{align*}
for all $j\in\N$.   Hence, going to a further subsequence if necessary, Lemma \ref{LemmaLowerBoundInterpolation} allows us to assume that
\begin{align*}
  \left\|\begin{bmatrix}
f(e_{1}/\kappa_n(X))-f(e_{2}/\kappa_n(X))&0&\ldots &0\\
f(e_{3}/\kappa_n(X))-f(e_{4}/\kappa_n(X))&0&\ldots &0\\ 
\vdots& \vdots &\ddots&\vdots\\
f(e_{2n-1}/\kappa_n(X))-f(e_{2n}/\kappa_n(X))&0&\ldots &0
 \end{bmatrix}\right\|_{\M_{n}(Y)}\geq \frac{\kappa_n(Y)}{2L\kappa_n(X)^\alpha}.
\end{align*}

Let $c_n\in \M_n(X)$ be the operator in $\cB(\ell_2)$ whose  $(j,1)$-coordinate is $e_{2j-1}/\kappa_n(X)$, for all $j\in \{1,\ldots, n\}$,   and all other coordinates are zero, and let $d_n\in \M_n(X)$ be the operator in $\cB(\ell_2)$ whose  $(j,1)$-coordinate is $e_{2j}/\kappa_n(X)$, for all $j\in \{1,\ldots, n\}$,   and all other coordinates are zero. So,   $\|c_n\|_{\M_n(X)}=\|d_n\|_{\M_n(X)}=1$ and  \eqref{Eq2.Thm.alpha.RC} gives  
\begin{align*}  
 \|f_n(c_n)-f_n(d_n)\|_{\M_{n}(Y)}\leq 2M.
 \end{align*}
 As  
\begin{align*}
 \|f_n(c_n)& -f_n(d_n)\|_{\M_{n}(X)}\\
 &=\left\|\begin{bmatrix}
f(e_{1}/\kappa_n(X))-f(e_{2}/\kappa_n(X))&0&\ldots &0\\
f(e_{3}/\kappa_n(X))-f(e_{4}/\kappa_n(X))&0&\ldots &0\\ 
\vdots& \vdots &\ddots&\vdots\\
f(e_{2n-1}/\kappa_n(X))-f(e_{2n}/\kappa_n(X))&0&\ldots &0
 \end{bmatrix}\right\|_{\M_{n}(Y)},
 \end{align*}
the arbitrariness of $n\in\N$ implies that  
 \[\frac{\kappa_n(Y)}{2L\kappa_n(X)^\alpha}\leq 2M\ \text{ for all }\ n\in\N.\]
  This finishes the proof. 
\end{proof}
  
  \begin{corollary}\label{CorRC}
Let $\theta,\gamma\in [0,1]$,
\begin{itemize}
    \item $X\in \{(R,C)_\theta,(\min(\ell_2),\max(\ell_2))_\theta,(R\cap C,R+C)_\theta\}$, and 
    \item $Y\in \{(R,C)_\gamma,(\min(\ell_2),\max(\ell_2))_\gamma,(R\cap C,R+C)_\gamma\}$.
\end{itemize}
  Then, 
$\alpha_{Y}^{\mathrm{ss}}(X)\geq \gamma/\theta$.
\end{corollary}

 \begin{proof}
   Corollaries \ref{CorollaryRCKappa} and \ref{CoroRcapCR+C} show that $\kappa_n(X)=n^{\theta/2}$ and $\kappa_n(Y)=n^{\gamma/2}$. Hence,  Theorem \ref{ThmLowerBoundAlpha} implies that $n^{\alpha^{\mathrm{ss}}_Y(X)\theta/2}\gtrsim n^{\gamma/2}$, so, $\alpha^{\mathrm{ss}}_Y(X)\geq \gamma/ \theta$.
 \end{proof}

 The next definition considers another approach to obtain lower bounds for the small scale distortion of maps between operator spaces. Similar definitions have already been studied by C. Rosendal and the first named author in the context of Banach spaces under the names of  \emph{uncollapsed} and \emph{almost uncollapsed} maps, see \cite{Rosendal2017Sigma,Braga2018JFA}.

  \begin{definition}
Let $X$ and $Y$ be operator spaces. We call a map $f\colon B_X\to Y$ \emph{completely almost uncollapsed} if there is $t\in(0,1)$ such that 
\[\|[x_{ij}]-[z_{ij}]\|_{\M_n(X)}= t \ \text{ implies }\ \|f_n([x_{ij}])-f_n([z_{ij}])\|_{\M_n(Y)}\geq \eps\]
for all $n\in\N$ and all $[x_{ij}],[z_{ij}]\in B_{\M_n(X)}$.
  \end{definition}

\begin{theorem}\label{ThmUpperBound}
Let $X$ and $Y$ be homogeneous Hilbertian operator spaces. If there is a   Lipschitz map $f\colon B_{X}\to Y$ which is  completely almost uncollapsed, then $\kappa_n(X)\lesssim  \kappa_n(Y)$.
\end{theorem}
  
  \begin{proof}
The proof resembles the one of Theorem \ref{ThmLowerBoundAlpha} but with the arguments for the  upper and lower estimates replacing each other. For this reason, we start this proof letting $X$, $Y$,  $(c_n)_n$ and $(d_n)_n$ be as in there. 

Since $f$ is completely almost uncollapsed, fix $t\in (0,1)$ and $\eps>0$ such that 
\[\|[x_{ij}]-[z_{ij}]\|_{\M_n(X)}= t \ \text{ implies }\ \|f_n([x_{ij}])-f_n([z_{ij}])\|_{\M_n(Y)}\geq \eps\]
for all $n\in\N$ and all $[x_{ij}],[z_{ij}]\in B_{\M_n(X)}$. Since $X$ is a homogeneous Hilbertian space, it immediately follows that  $\|c_n-d_n\|=\sqrt{2}$ for all $n\in\N$. Therefore, replacing each of the $c_n$'s and $d_n$'s by $(t/{\sqrt{2}})c_n$ and $(t/{\sqrt{2}})d_n$, respectively, we can assume that
\[
\|c_n-d_n\|=t\ \text{ for all }n\in\N.
\]
Our choice of $t$ and $\eps$ then give that
 \begin{equation}\label{Eq1LipUnc}
  \|f_n(c_n)-f_n(d_n)\|_{\M_n(Y)}\geq \eps\ \text{ for all }n\in\N.
  \end{equation} 
  
On the other hand, letting $L=\Lip(f)$, we have 
\[  \left\|f\left(\frac{t}{\sqrt{2}\kappa_n(X)}e_{2j-1}\right)-f\left(\frac{t}{\sqrt{2}\kappa_n(X)} e_{2j}\right)\right\|_Y\leq \frac{L}{\kappa_n(X)}\ \text{ for all }j,n\in\N.
  \]
Hence, proceeding as in the proof of Theorem  \ref{ThmLowerBoundAlpha} and  passing to a subsequence if necessary, Lemma \ref{LemmaLowerBoundInterpolation} gives   that  
\begin{equation}\label{Eq2LipUnc}
  \|f_n(c_n)-f_n(d_n)\|_{\M_n(Y)}\leq\frac{ 2 L\kappa_n(Y)}{\kappa_n(X)}\ \text{ for all }n\in\N .
  \end{equation} 
  Equations \eqref{Eq1LipUnc} and \eqref{Eq2LipUnc} together then imply that $\kappa_n(X)\lesssim  \kappa_n(Y)$ as desired. 
 \end{proof}
  
    \begin{corollary}\label{CorRC2}
Let $\theta,\gamma\in [0,1]$,
\begin{itemize}
    \item $X\in \{(R,C)_\theta,(\min(\ell_2),\max(\ell_2))_\theta,(R\cap C,R+C)_\theta\}$, and 
    \item $Y\in \{(R,C)_\gamma,(\min(\ell_2),\max(\ell_2))_\gamma,(R\cap C,R+C)_\gamma\}$.
\end{itemize}
 If there is a   Lipschitz map $f\colon B_{X}\to Y$ which is  completely almost uncollapsed, then $\theta\leq \gamma$.
\end{corollary}

 \begin{proof}
   Corollaries \ref{CorollaryRCKappa} and \ref{CoroRcapCR+C} show that $\kappa_n(X)=n^{\theta/2}$ and $\kappa_n(Y)=n^{\gamma/2}$. Hence,  Theorem \ref{ThmUpperBound} implies that $n^{\theta/2}\lesssim n^{\gamma/2}$, so, $\theta\leq \gamma$.
 \end{proof}
  
 \begin{proof} [Proof of Theorem \ref{ThmLipPreservatinOfKappa}]
   Suppose there is a completely Lipschitz in small scale map $f\colon B_X\to Y$ which is completely almost uncollapsed. In particular, $f$ is completely bounded in small scale and  $\alpha^{\mathrm{ss}}_Y(X)=1$. Therefore,  Theorem \ref{ThmLowerBoundAlpha} gives that  $\kappa_n(Y)\lesssim  \kappa_n(X)$ . On the other hand,  Theorem \ref{ThmUpperBound} implies   $\kappa_n(X)\lesssim  \kappa_n(Y)$. 
 \end{proof}

 \begin{proof}[Proof of Corollary \ref{CorRC3}]
This follows from Corollaries \ref{CorRC} and \ref{CorRC2}.
 \end{proof}

\section{A foray into the non Hilbertian setting}\label{SectionLocalized}

In this final section, we go  beyond the homogeneous Hilbertian case and provide lower bounds for compression exponents $\alpha_{Y}^{\mathrm{ss}}(X)$ for  non-Hilbertian operator spaces $X$.

Theorem \ref{ThmLowerBoundAlphaLocal} below is in a sense a localized version of Theorem \ref{ThmLowerBoundAlpha}. The only significant difference in its proof is that instead of using tools such as Rosenthal's $\ell_1$-theorem to extract a subsequence from an infinite sequence, we will extract a subsequence of a finite sequence using the following slight generalization of the original Bourgain-Tzafriri restricted invertibility theorem \cite[Theorem 1.2]{Bourgain-Tzafriri} (we point out that the version below follows e.g.\ from \cite[Theorem 2]{Spielman-Srivastava}).

\begin{theorem}\label{ThmBourgainTzafriri}
There exists a universal constant $D>0$ such that whenever $T : \ell_2^n \to \ell_2^n$ is a linear map with $\n{Te_j} \ge 1$  for each $1 \le j \le n$, where $\{e_j\}_{j=1}^n$ is the canonical basis of $\ell_2^n$, then there exists a subset $\sigma\subseteq\{1,2,\dotsc,n\}$ of cardinality $|\sigma| \ge Dn/\n{T}^2$ such that for any choice of scalars $\{a_j\}_{j\in\sigma}$ we have
\[
\bigg\| \sum_{j\in\sigma} a_j Te_j \bigg\| \ge D \bigg( \sum_{j\in\sigma} |a_j|^2\bigg)^{1/2}.
\]
\end{theorem}

For a homogeneous Hilbertian operator space $Z$ and $n\in\N$, we denote by $Z_n$ an $n$-dimensional subspace of $Z$ (they are all completely isometric by homogeneity, so there is no ambiguity).
Note that by Example \ref{ExampleRC}, in the statement of the following theorem the interval $[0,1/2]$ covers all possible values of the constant $c$.

\begin{theorem}\label{ThmLowerBoundAlphaLocal}
Let $Z$ and $Y$ be homogeneous Hilbertian operator spaces. Suppose that $\kappa_n(Y) \gtrsim n^{c}$ for some $c\in [0,1/2]$.
If $X$ is an operator space for which there exist a constant $A \ge 1$ and a sequence of injective linear maps $\varphi^n : Z_n \to X$ such that $\n{(\varphi^n)_n}\cdot \n{(\varphi^n)^{-1}} \le A$, then
$\kappa_n(Z)^{\alpha_Y^{\mathrm{ss}}(X)} \gtrsim n^{c/(1+2c)}$.
\end{theorem}

\begin{proof}
Without loss of generality, let us assume that for each $n\in\N$ we have $\n{(\varphi^n)_n} = 1$ and $\n{(\varphi^n)^{-1}} \le A$.
Suppose there is a map $f\colon  B_{X}\to Y$ which is  completely bounded in small scale and numbers $\alpha, L \ge 1$ such that 
\begin{equation}\label{Eq1.Thm.alpha.RC.local}
\|f(x)-f(y)\|_Y\geq \frac{1}{L}\|x-y\|_X^\alpha\ \text{ for all }\ x,y\in B_X.
\end{equation}
  As $f$ is  completely bounded in small scale, fix  $M>0$  such that
\begin{equation}\label{Eq2.Thm.alpha.RC.local}  
  \|f_n(a)\|_{\M_n(Y)}\leq M\ \text{  for all }\ a\in B_{\M_n(X)}.
\end{equation}

Fix $n\in\N$ and let $(e_j)_{j=1}^{2n}$ be an orthonormal basis for $Z_{2n}$.
For $1\le j \le 2n$, let $\widetilde{e}_j = \varphi^{2n}(e_{j})$, and note that $\| \widetilde{e}_j \| \le \n{e_j} = 1$.
Note that \eqref{Eq1.Thm.alpha.RC.local} implies that for $1 \le j \le n$,
  \begin{multline*}
 \left\| 
f\left( \frac{\widetilde{e}_{2j-1}}{\kappa_n(Z)}\right)-f\left(\frac{\widetilde{e}_{2j}}{\kappa_n(Z)}\right) \right\|_{Y}  
   \geq \frac{1}{L} \left\| 
 \frac{\widetilde{e}_{2j-1}}{\kappa_n(Z)} - \frac{\widetilde{e}_{2j}}{\kappa_n(Z)} \right\|^\alpha_{X} \\
 = \frac{1}{L\kappa_n(Z)^\alpha} \n{\widetilde{e}_{2j-1} - \widetilde{e}_{2j}}_X^\alpha 
\geq \frac{1}{L\kappa_n(Z)^\alpha A^\alpha} \n{e_{2j-1} - e_{2j}}_Z^\alpha \geq \frac{1}{L\kappa_n(Z)^\alpha A^\alpha}.
\end{multline*}

Let $c_n\in \M_n(X)$ be the matrix whose $(j,1)$-entry is $\widetilde{e}_{2j-1}/\kappa_n(Z)$, for $1\le j\le n$, and all other entries are zero, and let $d_n\in \M_n(X)$ be the matrix in $\M_n(X)$ whose $(j,1)$-entry is $\widetilde{e}_{2j}/\kappa_n(Z)$, for all $1\le j\le n$,  and all other entries are zero.
Since $ \n{(\varphi^{2n})_{2n}} = 1$, note that $\|c_n\|_{\M_n(X)} \le \kappa_n(Z)/\kappa_n(Z) = 1$.
Analogously, $\|d_n\|_{\M_n(X)}\le 1$.
Therefore, by \eqref{Eq2.Thm.alpha.RC.local}, we conclude that 
\begin{align*}  
 \|f_n(c_n)-f_n(d_n)\|_{\M_{n}(Y)}\leq 2M.
 \end{align*}
Letting $y_j = f(\widetilde{e}_{2j-1}/\kappa_n(Z))-f(\widetilde{e}_{2j}/\kappa_n(Z))$ for $1 \le j \le n$, the previous inequality means that
\begin{equation}\label{Eq3.Thm.alpha.RC.local}  
\left\|\begin{bmatrix}
y_1&0&\ldots &0\\
y_2&0&\ldots &0\\ 
\vdots& \vdots &\ddots&\vdots\\
y_n&0&\ldots &0
 \end{bmatrix}\right\|_{\M_{n}(Y)} \le 2M.
\end{equation}
Now, for any $c_1,c_2\dotsc,c_n \in \C$ it follows from Ruan's axioms that
\begin{multline*}
\n{\sum_{j=1}^n c_jy_j}_Y 
\le \\
\n{\begin{bmatrix}
c_1 & c_2 &\cdots &c_n\\
\end{bmatrix}}  \left\|\begin{bmatrix}
y_1&0&\ldots &0\\
y_2&0&\ldots &0\\ 
\vdots& \vdots &\ddots&\vdots\\
y_n&0&\ldots &0
 \end{bmatrix}\right\|_{\M_{n}(Y)}
\n{\begin{bmatrix}
1 \\ 0 \\ \vdots \\ 0\\
\end{bmatrix}}
 \le 2M\left( \sum_{j=1}^n |c_j|^2\ \right)^{1/2},
\end{multline*}
that is, the operator $T : Y_n \to Y$ which sends the $j$-th element of the canonical basis to $y_j$ has norm at most $2M$.
Since $\n{y_j}_Y \ge L^{-1}\kappa_n(Z)^{-\alpha} A^{-\alpha}$ for each $1\le j \le n$, it follows from Theorem \ref{ThmBourgainTzafriri} that there is a universal constant $D$ such that there is a subset $\sigma = \{\sigma_1, \dotsc, \sigma_m\} \subseteq \{1,2,\dotsc, n\}$ of cardinality $m \ge \frac{Dn}{4M^2}L^{-2}\kappa_n(Z)^{-2\alpha} A^{-2\alpha}$  such that the operator $T$, when restricted to the coordinate subspace corresponding to $\sigma$, is invertible and the norm of the inverse is at most $D^{-1} L\kappa_n(Z)^{\alpha} A^{\alpha}$. By homogeneity, the $\cb$-norm of the inverse of said restriction is also bounded by this same number.
Therefore,
\begin{equation}\label{Eq4.Thm.alpha.RC.local}  
\left\|\begin{bmatrix}
y_{\sigma_1}&0&\ldots &0\\
y_{\sigma_2}&0&\ldots &0\\ 
\vdots& \vdots &\ddots&\vdots\\
y_{\sigma_m}&0&\ldots &0
 \end{bmatrix}\right\|_{\M_{m}(Y)} \ge D L^{-1}\kappa_n(Z)^{-\alpha} A^{-\alpha} \kappa_m(Y),
\end{equation}
and thus we have from \eqref{Eq3.Thm.alpha.RC.local} and \eqref{Eq4.Thm.alpha.RC.local} that  
\[
2M \gtrsim \frac{\kappa_m(Y)}{\big( \kappa_n(Z) A \big)^\alpha} \gtrsim \frac{ (n \big( \kappa_n(Z) A \big)^{-2\alpha})^c }{\big( \kappa_n(Z) A \big)^\alpha} = \frac{n^c} {\big( \kappa_n(Z) A \big)^{\alpha(1+2c)}},
\]
where the implied constants are independent of both $n$ and $\alpha$.
It then follows that $\big( \kappa_n(Z) A \big)^{\alpha(1+2c)} \gtrsim n^c$, so $\big( \kappa_n(Z) A \big)^{\alpha} \gtrsim n^{c/(1+2c)}$ from where the desired result follows.
\end{proof}

\begin{remark}
If the space $Z$ in Theorem \ref{ThmLowerBoundAlphaLocal} satisfies  $\kappa_n(Z) \lesssim n^d$ for some constant $d$, we get the lower bound $\alpha_Y^{\mathrm{ss}}(X) \ge \frac{c}{d(1+2c)}$. However, this bound is trivial when $d\ge 1/4$: since $c \in [0,1/2]$ by Example \ref{ExampleRC}, we get $1 \ge \frac{c}{d(1+2c)}$. In particular, Theorem \ref{ThmLowerBoundAlphaLocal} gives no information when $Z = \mathrm{OH} = (R,C)_{1/2}$.   
\end{remark}

\begin{remark}\label{RemarkWeakCotype}
By \cite[Theorem 3.3]{Lee-weak-type-cotype}, the sequence of maps $\varphi^n : Z_n \to X$ in the hypotheses of Theorem \ref{ThmLowerBoundAlphaLocal} is guaranteed to exist whenever $X$ has weak cotype $(2,Z^*)$: we even get the stronger condition $\n{\varphi^n}_{\cb} \n{(\varphi^n)^{-1}} \le A$ for some constant $A$.
Such maps are called \emph{complete semi-isomorphisms} in the literature, see \cite[Section  3]{Oikhberg-Rosenthal}.
Since we will not need the aforementioned notion of weak cotype in this paper, the reader is directed to \cite{Lee-weak-type-cotype} for the definition. 
\end{remark}

As a first example of the consequences one can obtain from Theorem \ref{ThmLowerBoundAlphaLocal}, we state one that easily follows from our previous calculations of $\kappa_n$'s in specific cases.

\begin{corollary}
 Let $\theta,\gamma\in [0,1]$,
\begin{itemize}
    \item $Z\in \{(R,C)_\theta,(\min(\ell_2),\max(\ell_2))_\theta,(R\cap C,R+C)_\theta\}$, and 
    \item $Y\in \{(R,C)_\gamma,(\min(\ell_2),\max(\ell_2))_\gamma,(R\cap C,R+C)_\gamma\}$.
\end{itemize}
If $X$ is an operator space for which there exist a constant $A \ge 1$ and a sequence of injective linear maps $\varphi^n : Z_n \to X$ such that $\n{(\varphi^n)_n}\cdot \n{(\varphi^n)^{-1}} \le A$, (in particular, if $X$ has weak cotype $(2,Z^*)$), then $\alpha_Y^{\mathrm{ss}}(X) \ge \frac{\gamma}{(1+\gamma)\theta}$.
\end{corollary}

Now we present an example for some specific operator spaces $X$, namely the Schatten classes $S_p$. See \cite[Chapter 1]{Pisier_Lp} for the definition of their operator space structure.

\begin{corollary}
Let $1 \le p \le 2$, and let $Y$ be a homogeneous Hilbertian space such that $\kappa_n(Y) \gtrsim n^{c}$ for some $c\in [0,1/2]$. 
Then $\alpha_Y^{\mathrm{ss}}(S_p) \ge \frac{2p'c}{(1+2c)}$.
In particular, if $c>0$ then $\alpha_Y^{\mathrm{ss}}(S_1) = \infty$.
\end{corollary}

\begin{proof}
By \cite[Page 222]{lee-typeandcotype}, $S_p$ has cotype $\big(2,( R \cap C, R + C )_{1/p}\big)$, which implies weak cotype $\big(2,( R \cap C, R + C )_{1/p}\big)$. Note that $( R \cap C, R + C )_{1/p} = Z^*$ for $Z=( R \cap C, R + C )_{1/p'}$ by \eqref{Eq.DualOfInterpolatedSpace}, so by Remark \ref{RemarkWeakCotype} we can apply Theorem \ref{ThmLowerBoundAlphaLocal} to get $\kappa_n(Z)^{\alpha_Y^{\mathrm{ss}}(S_p)} \gtrsim n^{c/(1+2c)}$. But  $\kappa_n(Z) = n^{1/2p'}$ by Corollary \ref{CoroRcapCR+C}, yielding the conclusion. 
\end{proof}

Furthermore, we next show that Theorem \ref{ThmLowerBoundAlphaLocal} can always be applied to non-Hilbertian operator spaces.


\begin{corollary}\label{CorLowerBoundAlphaLocalDvoretzky}
Let $X$ be an infinite-dimensional operator space, and let $Z$ be a Dvoretzky space for $X$.
Let $Y$ be a homogeneous Hilbertian operator space such that $\kappa_n(Y) \gtrsim n^{c}$ for some $c\in[0,1/2]$. 
Then $\kappa_n(Z)^{\alpha_Y^{\mathrm{ss}}(X)} \gtrsim n^{c/(1+2c)}$.
In particular, if $X$ is a minimal operator space and $c>0$ then  $\alpha_Y^{\mathrm{ss}}(X) = \infty$.   

\end{corollary}

\begin{proof}
Since $Z$ is contained in an ultrapower of $X$, by the classical finite representability for ultrapowers of Banach spaces \cite[Proposition 6.1]{Heinrich},
for any $n\in\N$ and $\varepsilon>0$ we can find an injective linear map $\varphi^n : Z_n \to X$ such that $\n{(\varphi^n)^{-1}} \le 1$ and  $\n{\varphi^n} \le (1+\varepsilon)/n^2$.
A simple triangle inequality argument shows $\n{(\varphi^n)_n} \le 1+\varepsilon$.
Therefore, the hypotheses of Theorem \ref{ThmLowerBoundAlphaLocal} are satisfied which yields the  conclusion.

If $X$ is minimal, since both ultraproducts and subspaces of minimal operator spaces are also minimal, then the only possible (separable) Dvoretzky space for $X$ is $\min(\ell_2)$. Since $\kappa_n(\min(\ell_2))=1$ by Corollary \ref{CorollaryRCKappa}, we conclude ${\alpha_Y^{\mathrm{ss}}(X)}=\infty$.
\end{proof}

\begin{proof}[Proof of Corollary \ref{CorLowerBoundLipDvoretzky}]
Just as in the proof of Theorem \ref{ThmLipPreservatinOfKappa}, the existence of a completely Lipschitz in small scale embedding $B_{X}\to Y$ implies $\alpha_Y^{\mathrm{ss}}(X)=1$. Thus, the desired conclusion follows from Corollary \ref{CorLowerBoundAlphaLocalDvoretzky}.    
\end{proof}

As we have seen in the proof of Corollary \ref{CorLowerBoundAlphaLocalDvoretzky} above, for any infinite-dimensional operator space $X$ there is a homogeneous Hilbertian $Z$ which satisfies the condition in Theorem \ref{ThmLowerBoundAlphaLocal}. The opposite is also true: for a given homogenous Hilbertian operator space $Z$, it is not difficult to find nonhomogeneous and not Hilbertian operator spaces $X$ satisfying the desired condition (and not containing $Z$). For example, take the $\ell_p$-sum $\left( \bigoplus_{n=1}^\infty Z_n \right)_{\ell_p}$ for $p\in(1,\infty)\setminus\{2\}$, which obviously contains completely isometric copies of the $Z_n$ but is not homogeneous since it also contains a completely $1$-complemented copy of $\ell_p$, which is not homogeneous \cite[Page 137]{Junge-Habilitationschrift}.

\end{document}